\UseRawInputEncoding 
\documentclass[letterpaper,10pt]{amsart}
\usepackage{amsmath,amssymb,amsxtra,amsthm, amstext, amscd,amsfonts,fancyhdr, hyperref,enumerate}
\usepackage[OT2,T1]{fontenc}
\DeclareSymbolFont{cyrletters}{OT2}{wncyr}{m}{n}
\DeclareMathSymbol{\Sha}{\mathalpha}{cyrletters}{"58}

\newcommand{\bC}{{\mathbb{C}}}

\newcommand{\bN}{{\mathbb{N}}}

\newcommand{\bP}{{\mathbb{P}}}
\newcommand{\bQ}{{\mathbb{Q}}}
\newcommand{\bR}{{\mathbb{R}}}

\newcommand{\bZ}{{\mathbb{Z}}}

\newcommand{\Bx}{{\mathbf{x}}}
\newcommand{\By}{{\mathbf{y}}}

  \newcommand{\E}{{\mathcal{E}}}

\renewcommand{\L}{{\mathcal{L}}}
  
  \newcommand{\N}{{\mathcal{N}}}
\renewcommand{\O}{{\mathcal{O}}}

\renewcommand{\S}{{\mathcal{S}}}

  \newcommand{\X}{{\mathcal{X}}}

\newcommand{\fp}{\mathfrak{p}}

\newcommand{\fX}{\mathfrak{X}}

\newcommand{\BA}{\mathbf{A}}

\newcommand{\ord}{\operatorname{ord}}

\newcommand{\sqf}{\operatorname{sqf}}

\newcommand{\ep}{\varepsilon}

\newcommand{\upchi}{{\raise.35ex\hbox{$\chi$}}}

\makeatletter
\@namedef{subjclassname@2010}{%
  \textup{2010} Mathematics Subject Classification}
\makeatother

\newtheorem{theorem}{Theorem}[section]
\newtheorem{corollary}[theorem]{Corollary}
\newtheorem{proposition}[theorem]{Proposition}
\newtheorem{lemma}[theorem]{Lemma}

\theoremstyle{definition}
\newtheorem{definition}[theorem]{Definition}
\newtheorem{question}[theorem]{Question}

\numberwithin{equation}{section}

\begin{document}

\title{The density of rational points on $\bP^1$ with three stacky points}

\author{Brett Nasserden}
\address{Department of Pure Mathematics \\
University of Waterloo }
\email{bnasser@uwaterloo.ca}
\indent

\author{Stanley Yao Xiao}
\address{Department of Mathematics \\
University of Toronto \\
Bahen Centre \\
40 St. George Street, Room 6290 \\
Toronto, Ontario, Canada \\  M5S 2E4 }
\email{syxiao@math.toronto.edu}
\indent


\begin{abstract} In this paper we consider the density of rational points on the ``stacky" curve $\fX(\bP^1;0,2;1,2;\infty,2)$ which is $\bP^1$ with three half points, with respect to the so-called Ellenberg-Satriano-Zuerick-Brown height. In particular, we prove a conjecture of Ellenberg. 
\end{abstract}

\maketitle

\section{Introduction} 

Two of the outstanding conjectures in number theory are the so-called Manin-Batyrev conjecture \cite{BatmanConj} for the density of rational points on Fano varieties, and Malle's conjecture \cite{MalleConj} on the number of number fields of bounded discriminant having fixed degree and Galois group. Both conjectures assert, roughly, that the objects to be counted satisfy an asymptotic formula of the form 
\[C \cdot X^\alpha (\log X)^\beta, \]
where $C, \alpha, \beta$ are non-negative numbers with $C, \alpha > 0$ which can be computed explicitly within their respective conjectural frameworks. \\

In a forthcoming article J.~Ellenberg, M.~Satriano, and D.~Zuerick-Brown formulate a bold conjecture that has both the Manin and Malle conjetures as special cases  \cite[Main Conjecture]{ESD-B}. Their conjecture concerns counting rational points with respect to a new theory of heights applying broadly to \emph{algebraic stacks}. While the Manin and Malle conjectures are well studied, very little is known about the general conjecture. \\

In the framework of \cite{ESD-B} the Malle and Manin conjectures represent two extremes of their theory of heights. The Manin conjecture involves counting points on a projective variety with respect to a Weil height and no theory of algebraic stacks is required. On the other hand from the point of view of \cite{ESD-B} the Malle conjecture involves counting rational points on the classifying stack $BG$ where $G$ is a finite group. The theory of algebraic stacks is essential for this interpretation of the Malle conjecture, and the standard theory of heights on projective algebraic varieties is insufficient in this case.\\

In this article we consider instances of the main conjecture in \cite{ESD-B} that lie between the two extremes described above. In other words, we consider cases that involve mixing the stacky and non-stacky phenomena. We formulate a new theory of heights on a stacky analogue of smooth projective algebraic curves and show that this stacky height is enough to determine the integral points on these "stacky curves". We then consider a particular stacky curve suggested by J.~Ellenberg\footnote[1]{"What?s up in arithmetic statistics?" Number Theory Web Seminar, July 23, 2020} and show that our theory of heights matches \cite{ESD-B} in this instance. Finally, we verify a specific instance of the main conjecutre in \cite{ESD-B} given by Ellenberg\footnotemark[1]  using analytical methods.\\ 

Our point of view with algebraic stacks is to adopt a bottom up perspective. In other words, to define our algebraic stacks in terms of a base variety along with some extra data which is enough to construct a unique algebraic stack. As we are interested in a well behaved family of stacky curves this description will be particularly simple. The reason for this choice is that although the heights we are concerned with are motivated by stacky phenomena, the analysis of the heights we are interested in lie firmly within the purview of analytic number theory and so a relatively simple exposition is preferable. The bottom up point of view allows us to discuss the objects we are interested in a concrete way that avoids technicalities and emphasizes the data most important for our purposes. The interested reader may consult \cite{BottomUp} for general results involving the bottom up perspective on algebraic stacks and \cite[Lemma 5.3.10]{CanRing} for the case of stacky curves.\\

We now describe the stacky curve we are most interested in. Those unfamiliar with the theory of algebraic stacks may note the explicit form of our height given in (\ref{eq:HeightDef}) and move on to the definition of an $M$-curve given by Definition \ref{def:MCurve}. Let $\fX(\mathbb{P}^1_\bQ;0,2;1,2;\infty,2)$ be the algebraic stack obtained by replacing $\{0,1,\infty\}\subseteq {\bP^1}$ with $B(\bZ/2 \bZ)$. For a precise definition of this object, see \cite{BottomUp}. The points $0,1,\infty$ are "stacky points" which have an attached stabilizer group $\mathbb{Z}/2\mathbb{Z}$ and account for $\fX(\mathbb{P}^1_\bQ;0,2;1,2;\infty,2)$ not being a projective variety. The other points behave like the points on the quasi-projective variety $\bP^1-\{0,1,\infty\}$.  This construction is an example of a "bottom-up" description of an algebraic stack; the algebraic stack $\fX(\mathbb{P}^1_\bQ;0,2;1,2;\infty,2)$ is constructed from the data of $\mathbb{P}^1_\bQ$ and the points with their associated multiplicities. In fact a large class of algebraic stacks can be constructed in this way, see \cite{BottomUp} for the details and for further references and recent appearances of these objects see \cite{CanRing}, \cite{LGStackyCurve}, and \cite{StackyCurves}. \\ 

One of the novelties of \cite{ESD-B} is that \emph{vector bundles} have an associated height, while in the classical setting a height is only associated to a line bundle.  Therefore to apply the theory of \cite{ESD-B} one must choose a vector bundle on $\fX$. We choose the tangent line bundle $\mathbb{T}_{\fX}$.  There is a birational mapping $\pi\colon \fX\rightarrow \bP^1$ called the \emph{coarse space} map. The height associated to the tangent bundle $\mathbb{T}_\fX$ described by \cite{ESD-B} can be described in terms of the coarse space map as follows. Given a point $P\in \fX(\mathbb{Q})$ we have the associated point $\pi(P)=(a_P:b_P)\in \bP^1$. Now consider the height $H(a,b)$ on primitive integer pairs $(a,b)$ by
\begin{equation}\label{eq:HeightDef}
    H(a,b) = \sqf(a) \sqf(b) \sqf(a+b) \max\{|a|,|b|\},
\end{equation}
with $\sqf(n) = n/k^2$, where $k^2$ is the largest square dividing $n$. Then given $P\in \fX(\mathbb{Q})$ we define $H_{\mathbb{T}_{\fX}}(P)=H(\pi(P))$ by choosing an integral representation for $\pi(P)$.\\ 

Notice that the height (\ref{eq:HeightDef}) can be defined in terms of the coarse moduli space $\bP^1$ of $\fX(\mathbb{P}^1_\bQ;0,2;1,2;\infty,2)$ and a formula that takes into account the multiplicities in a simple way. Motivated by this we now eschew the theory of algebraic stacks in favor of Darmon's $M$-curves, which is an essentially equivalent theory that emphasizes the bottom up perspective to algebraic stacks. In other words, we keep track of the minimum amount of data that can be used to construct the algebraic stack. One may think of this as being analogous to only keeping track of a particular Weierstrass equation of an elliptic curve. 

\begin{definition}[\cite{MCurves}]\label{def:MCurve} Let $K$ be a number field.
 An \textbf{$M$-curve over $K$} consists of the following data: 
 
 \begin{itemize}
     \item A smooth projective curve $X$ defined over a number field $K$, and 
     \item For each $P\in X(K)$ a multiplicity $m_P\in \mathbb{Z}_{\geq 1}\cup\{\infty\}$ with $m_P=1$ for all but finitely many $P$.
 \end{itemize} 
 
 We use the notation.
	\[\X=(X;P_1,m_1;P_2,m_2;...;P_r,m_r)\] to denote the $M$-curve with multiplicities $m_{P_i}=m_i$ and if $Q\notin\{P_1,...,P_r\}$ then $m_Q=1$. 
\end{definition}

The $M$-curve we are most interested is given by $\bP^1_{2,2,2}=\X(\bP^1_\bQ;0,2;1,2;\infty,2)$ with the height on $\bP^1$ given by (\ref{eq:HeightDef}). Our main goal is to count rational points on $\bP^1$ with respect to the height (\ref{eq:HeightDef}). On writing
\[a = x_1 y_1^2, b = x_2 y_2^2, x_1, x_2 \text{ square-free}\]
We then have
\[H(a,b) = x_1 x_2 \sqf(x_1 y_1^2 + x_2 y_2^2) \max\{|x_1 y_1^2|,  |x_2 y_2^2|\}.\]
and the max on the right hand side is dependent only on the relative size of $|a|,|b|$. If we write
\begin{equation} \label{variety} x_1 y_1^2 + x_2 y_2^2 = x_3 y_3^2,\end{equation} 
then we further obtain the expression
\[H(a,b) = \max\{|x_2 x_3 (x_1 y_1)^2|, |x_1 x_3 (x_2 y_2)^2|\}.\]
We may assume without loss of generality that $|x_1 y_1^2| \geq |x_2 y_2|^2$ and $x_1 > 0$, so that 
\[H(a,b) = |x_2 x_3 (x_1 y_1)^2|.\]

We put
\begin{equation} \label{NT} 
N(T) = \# \{\Bx, \By \in \bZ_{\ne 0}^3 : \gcd(x_i, x_j), \gcd(y_i, y_j) = 1 \text{ for } i \ne j, x_i \text{ square-free for } i = 1,2,3,
\end{equation}
\[x_1 y_1^2 + x_2 y_2^2 = x_3 y_3^2, x_1 > 0, x_1 y_1^2 \geq |x_2 y_2^2|, |(x_1 y_1)^2 x_2 x_3| \leq T\}\]

Our first main result will be the following:

\begin{theorem} \label{MT} There exist positive numbers $c_1, c_2, c_3$ such that 
\[c_1 T^{1/2} (\log T)^3 < N(T) < c_2 T^{1/2} (\log T)^3 \]
for all $T > c_3$. 
\end{theorem}

In particular, we confirm Ellenberg's conjecture\footnotemark[1] that $N(T) = O_\ep \left(T^{1/2 + \ep}\right)$. Indeed, our theorem gives an exact order of magnitude for $N(T)$. \\

We also propose a general height applicable to $M$-curves with finite multiplicities (in particular to stacky curves in the sense of \cite[Definition 5.2.1]{CanRing}) which matches (\ref{eq:HeightDef}) for our specific curve. Moreover, our height has the remarkable feature that it naturally detects the difference between \emph{integral} and \emph{rational} points on $M$-curves (here we are using Darmon's notion of integrality; see Definition \ref{DarInt})

\begin{theorem} \label{height thm} For each $M$-curve $\X/\bQ$ with finite multiplicities there exists a height $H = H_\L H_\X$ with the property that an element of $x \in \X(K)$ is integral if and only if $H_\X(x) = 1$. Moreover, for $\X = \X(\bP^1;0,2;-1,2;\infty,2)$ the height $H_\L H_{\X}$ is equivalent to the height given by (\ref{eq:HeightDef}).
\end{theorem} 

Theorem \ref{height thm} is a consequence of the more technical Theorem \ref{height MT}. \\

We illustrate how Theorem \ref{height thm} allows one to detect integral points on $M$-curves. In this case the standard height is given by $H_s(a,b) = \max\{|a|,|b|\}$ and the stacky height given by (\ref{eq:HeightDef}). They are equal precisely when
\[|\sqf(a) \sqf(b) \sqf(a+b)| = 1, \]
or in the notation of (\ref{variety}), that $|x_1| = |x_2| = |x_3| = 1$. (\ref{variety}) then turns into 
\[\pm y_1^2 \pm y_2^2 = \pm y_3^2, \]
and up to rearranging we are essentially counting points on the conic 
\begin{equation} \label{intcon} y_1^2 + y_2^2 = y_3^2.\end{equation}
Therefore if we denote by $\N(T)$ the number of integral points (in the sense of Definition \ref{DarInt}) on $\bP_{2,2,2}^1$ then:

\begin{corollary} \label{intct}  There exist positive numbers $c_1, c_2, c_3$ such that for all $T > c_3$ we have \[c_1 T^{1/2} < \N(T) <c_2 T^{1/2}.\]
\end{corollary}

The proof is elementary, since the curve can be explicitly parametrized by 
\[y_1 = u^2 - v^2, y_2 = 2uv, y_3 = u^2 + v^2.\]
The condition $\max\{|y_1|, |y_2\} \leq T^{1/2}$ is subsumed by $u^2 + v^2 \leq 4T^{1/2}$ say, so number of possible $u,v$'s is $\asymp T^{1/2}$ as desired. \\

Theorem \ref{MT} and Corollary \ref{intct} imply that asymptotically $0$-percent of the rational points on $\bP_{2,2,2}^1(\bQ)$ are integral, in the sense of Darmon (Definition \ref{DarInt}). 

\subsection{Organization of the paper} The proofs of Theorems \ref{MT} and \ref{height thm} (Theorem \ref{height MT}) are essentially disjoint, and are contained in Sections \ref{sec:MCurves} and \ref{count sec} respectively. The reader interested in one but not the other can essentially read these sections independently of each other. \\

We note that in the proof of Theorem \ref{MT} we shall require a counting result on the number of diagonal ternary quadratic forms having bounded, square-free discriminant given as Proposition \ref{quad form prop}. This may be of independent interest. We further note that it is possible to refine the arguments given in Section \ref{count sec} to give an asymptotic formula in Proposition \ref{quad form prop}, but we do not pursue this in the present paper. 

\subsection*{Notation} We denote by $d_k(n)$ for the number of ways of writing $n$ as a product of $k$ (not necessarily distinct) positive integers, and write $d(n) = d_2(n)$ for the usual divisor function. We will also use the big-$O$ notation as well as Landau's notation. In particular, we will denote in the subscripts any dependencies; if there are no subscripts, then the implied constants are absolute. 

\section{Heights on $M$-curves}\label{sec:MCurves}

We now define our height function on a $M$-curve $\X$. Before continuing let us fix some notation. We assume that $\X=(X;P_1,m_1;...;P_r,m_r)$ is an $M$-curve as in Definition \ref{def:MCurve}. We furthermore assume that $1<m_{P_i}<\infty$.

\subsection{Translation between $M$-curves and stacky curves} Before moving on we explain the connection between stacky curves and $M$-curves. Those uninterested may skip this section and safely work with $M$-curves. Given a nice stacky curve $\fX$ over a number field $K$ there is a morphism $\pi\colon \fX\rightarrow X$ to a curve $X$ called the coarse space morphism, which is the universal morphism from $\fX$ to a scheme. In practice one often constructs $\fX$ from $X$ by specifying a collection of points $P_1,...,P_r$ in the coarse space $X$ and attaching stabilizer groups $\mu_{P_i}$ to each $P_i$. This is the bottom up approach of constructing an algebraic stack. One can think of this as specifying the ramification data of the coarse space morphism $\pi\colon \fX\rightarrow X$. We often think of these points as "fractional points" because in divisor class group of the associated curve we have added the point $\frac{1}{\#\mu_{P_i}}P_i$. In other words, we think of a stacky curve as a smooth curve $X$ with a choice of points $P_1,...,P_r$ with stabilizer groups $\mu_{P_i}$ attached to each $P_i$.  This data defines an $M$-curve $\X=(X;P_1,\#\mu_{P_1};...;P_r,\#\mu_{P_r})$. Conversely, given an $M$-curve $\X=(X;P_1,m_1;...;P_r,m_r)$ with each $1<m_{P_i}<\infty$ we consider the stacky curve given with points $P_i$ having the stabilizer group $\mu_{m_{P_i}}$. In this way one may establish a bijection between smooth proper geometrically connected Delign-Mumford stacks of dimension 1 over $K$ with stacky points defined over $K$ that contain an open dense subscheme and possess a projective coarse moduli space and $M$-curves over $K$ with finite multiplicities.

\subsection{Construction of heights}\label{subsec:HeightsMCurves}

Notice that the height (\ref{eq:HeightDef}) is defined in terms of a function on $\bP^1$. Since a stacky curve and its coarse space agree up to a finite set of points and we are interested in counting points asymptotically it suffices to count points on the coarse space with the height (\ref{eq:HeightDef}). With this inspiration our goal is to take a stacky curve $\fX$ with coarse space $X$ and construct the associated $M$-curve $\X=(X;P_1,m_1;...;P_r,m_r)$. We then define a height on the rational points of $X$ using the $M$-curve data with the purpose of counting points with respect to this height. \\

Choose a finite set of primes $S$ of $\O_K$ containing all the primes of bad reduction along with a smooth and proper model $\underline{X}$ of $X$ over $\O_{K,S}$.  \emph{Everything we do is relative to this choice of model}, similar to how everything we do is relative choosing the finite set of primes $S$. We will use the following notion of \emph{intersection multiplicity} from \cite{MCurves} to define integral points on $M$-curves and to define our heights. 

\begin{definition}[\cite{MCurves}] 
Let $P,Q$ be distinct points in $X(K)$ and place $\nu$ a place in $K$ with $\nu\notin S$. Take $\fp_\nu \subset \O_K$ to be the prime ideal associated to $\nu$. We define the \textbf{intersection multiplicity} of $P$ and $Q$ at $\nu$ as follows.
	
	\[(P\cdot Q)_\nu:=\max\{m:\textnormal{ the images of }P,Q
	\textnormal{ in }\underline{X}(\O_{K,S}/\fp_\nu^m)\textnormal{ are equal.}\}\]
where the maximum over the empty set is defined to be 0 above.	 
\end{definition}

We wish to define a height function on the $M$-curve which takes into account the local multiplicities.  Motivated by the work of \cite{ESD-B} our strategy will be to define a height function that takes into account a global classical height on the base curve $X$ and a local part that depends on the intersection multiplicities. In the case relevant to us we recover (\ref{eq:HeightDef}). \\

Let us fix an $M$-curve $\X=(X;P_1,m_1;...;P_r;m_r)$. To avoid complications with the infinite places we shall assume that $S$ contains all infinite places of $K$. Our height function will have a classical part which only depends on the underlying curve $X$ and a stacky part that depends on the points and multiplicities. To define the classical part of our height we choose an ample line bundle $\L$ on $X$ and a multiplicative ample height $H_\L$. Our strategy for the stacky contributions is to take into account all of the primes that are not in $S$. Since our intersection multiplicities depend on the choice of $S$ and model $\underline{X}$, so will our heights. Let $P$ be a point of $X$. We define 
\begin{equation}\label{eq:LocalLambdaFactor}
\lambda_{S,\underline{X},\nu}(P,t)=\lambda_{\nu}(P,t)=\textnormal{N}(\fp_\nu)^{(t\cdot P)_\nu}    
\end{equation} for $\nu\notin S$ and set
\begin{equation}\label{eq:LambdaFactor}
\lambda(P,t)=\prod_{\nu\notin S}\lambda_\nu(P,t).
\end{equation}

 To take into account the multiplicities of points we only consider $\lambda(P,t)$ up to $m_P$-powers. That is we look at the image  $\overline{\lambda(P,t)}\in \bQ/\bQ^{m_P}$ and consider this to be the local contribution to the height. Precisely we define

\begin{equation}\label{eq:PMHeight}
H_{S,\underline{X}}(P,t)=H(P,t)=m_P\textnormal{-free part}(\lambda(P,t))    
\end{equation}

and $H(P,t)=1$ if $m_P=1$.\\

We now multiply the classical and stacky contributions to obtain our height function. To sum up, given an $M$-curve $\X=(X;P_1,m_1;...;P_r,m_r)$ and a choice of an ample height $H_\L$ on $X$, a choice of primes $S$ and a model $\underline{X}$ we define

\begin{equation}\label{eq:MHeight}
H_{\X,S,\underline{X},\L}(t)=H_\L(t)\prod_{i=1}^rH_{S,\underline{X},}(P_i,t)=H_\L(t)\prod_{i=1}^rm_{P_i}\textnormal{-free part}(\lambda(P_i,t))
\end{equation}
We call $H_\L(t)$ the \emph{classical} part and $\prod_{i=1}^rH_{S,\underline{X},}(P_i,t)$ the \emph{stacky part} of the height $H_{\X,S,\underline{X},\L}(t)$.\\

What is interesting is that our height, given by (\ref{eq:PMHeight}) below, allows us to differentiate between \emph{rational} and \emph{integral} points on $M$-curves; this is one of the main features of our height and thus provide additional support that Ellenberg, Satriano, and Zuerick-Brown's theory of heights on algebraic stacks is an appropriate one. This is expanded in the subsection below. 

\subsection{Integral Points on $M$-Curves}
Here we show that the the height (\ref{eq:MHeight}) can be used to obtain information about integral points on $\mathcal{X}$. In particular, the set of integral points is contained in the set of points where (\ref{eq:MHeight}) is equal to 1. When we take $K=\bQ$ we see that this condition is sufficient. In other words the $S$-integral points are those where the stacky part of the height is trivial. Following Darmon \cite{MCurves} we have the following notion of integral points on an $M$-curve:
\begin{definition}[Darmon] \label{DarInt} 

Let $\X=(X;P_1,m_1;...;P_r,m_r)$ be a $M$-curve over a number field $K$, $S$ a finite set of places of $K$ containing all primes of bad reduction. Let $\underline{X}$ be a smooth proper model for $X$ over $\O_{K,S}$. The $(\underline{X},S)$-integral points of $\X$ (usually abbreviated to $S$-integral points of $\X$) are the points $t\in X(K)$ such that

\begin{equation}\label{eq:SInt}
(t\cdot P)_\nu\equiv 0\mod m_P    
\end{equation}
for all $P\in X(K)$ and  $\nu\notin S$. 

\end{definition}

We shall prove the following theorem:

\begin{theorem} \label{height MT}
Let $\X=(X;P_1,m_1;....;P_r,m_r)$ be an $M$-curve over $K$ satisfying our assumptions and choose $S$ and a model $\underline{X}$ as we have specified. Then we have the following conclusions.

\begin{enumerate}
    \item \[\X(\O_{K,S,\underline{X}})\subseteq \bigcap_{m_P>1}X(P;K)\]
    where  $\X(P;K)=\{t\in X(K)\colon H(P,t)=1\}$.
    
    \item If $K=\bQ$ then
    \[\X(\O_{K,S,\underline{X}})= \bigcap_{m_P>1}X(P;K).\]
    In particular, the set of $S$-integral points of $\X$ is precisely the set of points where $H(P,t)=1$ for all $P$ with $m_P>1$. 
    
\end{enumerate}

\end{theorem} 
Fix a prime $\nu\notin S$ and write $(t\cdot P)_\nu=m_{P}^{e_{\nu,P}(t)}\cdot q_{\nu,P}(t)$ where $e_{\nu,P}(t)\geq 0$ and $q_{\nu,P}(t)\geq 0$ is not divisible by $m_P$. In other words $q_{\nu,P}(t)$ is the $m_P$-free part of $(t\cdot P)_\nu$. Set $\textnormal{N}(\fp_\nu)=p_\nu^{f(\nu)}$. Then \begin{equation}\label{eq:locterms}
    \lambda_\nu(P,t)=p_\nu^{f(\nu)(t\cdot P)_\nu}=p_\nu^{m_{P}^{e_{\nu,P}(t)}\cdot q_{\nu,P}(t)\cdot f(\nu)}
\end{equation}
and
\begin{equation}\label{eq:expanded}
\lambda(P_i,t)=\prod_{\nu\notin S}p_\nu^{m_{P_i}^{e_{\nu,P_i}(t)}\cdot q_{\nu,P_i}(t)\cdot f(\nu)}
\end{equation}
Using the functions $\lambda(P,t)$ we can find subsets of the rational points that contain all integral points. 
\begin{proposition}\label{prop:cuttingout}
Suppose that $m_P>1$. Define $\X(P;K)=\{t\in X(K)\colon H(P,t)=1\}$. Then \[\X(\O_{K,S,\underline{X}})\subseteq \X(P;K).\]
\end{proposition}
\begin{proof}
Suppose that $t$ is an $S$-integral point. Then $(t\cdot P)_\nu\equiv 0\mod m_{P}\Rightarrow e_{\nu,P}(t)>0$ for all $\nu\notin S$. Thus 

\[\lambda(P,t)=\prod_{\nu\notin S}p_\nu^{m_{P_i}^{e_{\nu,P_i}(t)}\cdot q_{\nu,P_i}(t)\cdot f(\nu)}=\left(\prod_{v\notin S}p_\nu^{m_{P_i}^{e_{\nu,P_i}(t)-1}\cdot q_{\nu,P_i}(t)\cdot f(\nu)}\right)^{m_P},\]
whence $H(P,t)=1$ as $\lambda(P,t)$ is an $m_P$-power.
\end{proof}
We see that each point $P$ with multiplicity $m_P>1$ imposes a height dropping condition on the set of integral points. Thus to study integral points it suffices to study

\[\X(\O_{K,S,\underline{X}})\subseteq \bigcap_{m_P>1}X(P;K).\]

There is the following easy consequence that may be useful if one has access to the height $H_{\X,S,\underline{X},\L}$ but perhaps not the local factors.

\begin{corollary}
Let $H_\L$ be an ample height with $H_\L(t)>0$ for all $t\in X(K)$.Then

\[\X(\O_{K,S})\subseteq \{t\in X(K)\colon H_{\X,\L}(t)=H_{\L}(t)\}.\]

\end{corollary} 

The integral points are contained in the locus where the height can be computed classically, in other words where the stacky part of the height is trivial. The  difference in these sets can be potentially explained by interactions between $f(\nu)=[\O_K/\fp_\nu\colon \bZ/p_\nu]$ and the integers $m_P$ and the splitting of the primes $p_\nu$ in $K$. 

\begin{proof}[Proof of Theorem \ref{height MT}]
We have already shown part $(1)$ of Theorem \ref{height MT} in (\ref{prop:cuttingout}). We turn to part $(2)$ and assume that $K=\bQ$. We know that that $\bigcap_{m_P>1}X(P;\bQ)\subseteq \X(\O_{\bQ,S,\underline{X}})$ by (\ref{prop:cuttingout}). We now show the reverse inclusion. Let $t\in X(\bQ)$ with $H(P,t)=1$ for all $P$ with $m_P>1$. Since $K=\bQ$ we have that $\textnormal{N}(\fp_\nu)=p_\nu$ and $f(\nu)=1$ for all finite places $\nu$. Fix $P$ with $m_P>1$. Towards a contradiction suppose that $(t\cdot P)_{\nu_0}\neq 0 \mod m_P$ for some $\nu_0\notin S$.  Then $e_{\nu_0,P}(t)=0$. Notice that $H(P,t)=1$ means that $\lambda(P,t)$ is an $m_P$-power. Since if $\nu\neq \nu^\prime$ we have that $p_\nu\neq p_{\nu^{\prime}}$ we have by unique factorization of integers that

\[\lambda(P,t)=\prod_{\nu\notin S}p_\nu^{m_{P}^{e_{\nu,P}(t)}\cdot q_{\nu,P}(t)}=(\prod_{\nu\notin S}p_\nu^{z_\nu(t)})^{m_P}\]

for some integers $z_\nu(t)$. In particular for $\nu_0$ we have 

\[p_{\nu_0}^{m_{P}^{e_{\nu_0,P}(t)}\cdot q_{\nu_0,P}(t)}=p_{\nu_0}^{{q_{\nu_0,P}(t)}}=p_{\nu_0}^{z_{\nu_0}(t)m_P}\]

Thus $z_{\nu_0}(t)m_P=q_{\nu_0,P}(t)$ which contradicts $q_{\nu_0,P}(t)$ being indivisible by $m_P$. Thus for all $m_P>1$ and $\nu\notin S$ we have $(t\cdot P)_\nu\equiv 0 \mod m_P$ and $t$ is an $S$-integral point of $\X$ by definition.
\end{proof}

In the next subsection we demonstrate our height in the simplest cases, where the base curve is $\bP^1$ and we have three points with multiplicity exceeding one. 

\subsection{Main example - $\bP^1$ with three points of large multiplicity}\label{subsec:mainex}
Now we apply these definitions in the case relevant to the conjecture of \cite{ESD-B}. Let $X=\bP^1_\bQ$ and $S=\{\nu_\infty\}$ and take $\L$ to be $\O_{\bP^1}(1)$ so the ample height is the usual one. Now we define the $M$-curve

\[\bP^1_{p,q,r}:=(\bP^1_\bQ;0,p;-1,q;\infty,r)\]
where $-1$ is used instead of 1 to match the constructions in \cite{ESD-B}. Let $t=[a:b]\in \bP^1(\bQ)-\{\infty\}$ with $a,b$ coprime integers. Then we have that

\[(t\cdot 0)_\nu=\ord(a)_\nu,\ (t\cdot (-1))_\nu=\ord_\nu(a+b),\ (t\cdot\infty)_\nu=\ord_\nu(b)\]
for all finite primes $\nu$. The product formula gives in our specific case gives that \[\lambda(0,t)=\mid a
\mid,\lambda(-1,t)=\mid a+b
\mid,\textnormal{ and, }\lambda(\infty,t)=\mid b
\mid.\]
Now we consider our points up to $m_P$-powers. That is when $t=[a:b]$ with $a,b$ non-zero coprime integers we consider the image 

\[\overline{\lambda(P,t)}\in \bQ^*/(\bQ^*)^{m_P}\] 
and consider this to be the local contribution to the height of $t$ at $P$. The contribution of these local heights and the global height now gives for $t=[a:b]$ with $a,b$ coprime and non-zero integers that

\begin{align*}
H_{\bP^1_{p,q,r}}([a:b])&=H_{\bP^1_{p,q,r}}(0,[a:b])H_{\bP^1_{p,q,r}}(-1,[a:b])H_{\bP^1_{p,q,r}}(\infty,[a:b])H_{\bP^1_\bQ}([a:b])\\&=p\textnormal{-free part}(\mid a\mid)\cdot q\textnormal{-free part}(\mid a+b\mid)\cdot r\textnormal{-free part}(\mid b\mid)\cdot \max(\mid a\mid,\mid b\mid )
\end{align*}
Taking $p=q=r=2$ we obtain 

\[H_{\bP^1_{2,2,2}}([a:b])=2\textnormal{-free part}(\mid a\mid)\cdot 2\textnormal{-free part}(\mid a+b\mid)\cdot 2\textnormal{-free part}(\mid b\mid)\cdot \max(\mid a\mid,\mid b\mid )\]
which is the desired height function of (\cite{ESD-B}). \\

These discussions lead to the following question: \\

\emph{How many points $[a:b]\in \bP^1(\bQ)$ satisfy}

\[p\textnormal{-free part}(\mid a\mid)\cdot q\textnormal{-free part}(\mid a+b\mid)\cdot r\textnormal{-free part}(\mid b\mid)\cdot \max(\mid a\mid,\mid b\mid )\leq T?\]

More generally:\\

\emph{Let $\X=(X;P_1,m_1;...;P_r,m_r)$ be a $M$-curve over a number field $K$, $S$ a finite set of places of $K$ containing the infinite places and all primes of bad reduction. Let $\underline{X}$ be a smooth proper model for $X$ over $\O_{K,S}$. Then how many points $t\in X(K)$ satisfy}

\[H_{\X,S,\underline{X},\L}(t)\leq T\]
\emph{Can one find an asymptotic  formula for the number of such points?} \\

The rest of this article is dedicated to analyzing the case $p=q=r=2$; that is, to the proof of Theorem \ref{MT}.

\section{Counting rational points on $\X_\bQ = (\bP_\bQ^1; 0,2;-1,2;\infty,2)$}
\label{count sec}

In this section, we prove Theorem \ref{MT}. To do so we will show that $N(T) = O \left(T^{1/2} (\log T)^3 \right)$ and give a separate argument to show that $N(T) \gg T^{1/2} (\log T)^3$. The incompatibility of these two arguments represents the main obstacle as to why an asymptotic formula for $N(T)$ remains elusive. \\ 

We consider the problem of counting integral points on the variety defined by (\ref{variety}), subject to the constraint
\begin{equation} \label{ht bd} 0 < |x_2 x_3 (x_1 y_1)^2| \leq T, |x_1 y_1^2| \geq |x_2 y_2^2|.  \end{equation} 

To obtain the upper bound we must dissect (\ref{ht bd}) into suitable ranges. When $|x_1 x_2 x_3| \leq T^{1/2}$ we fix $x_1, x_2, x_3$ and treat (\ref{variety}) as a diagonal ternary quadratic form, say $Q_\Bx$. It is then the case that 
\begin{equation} \label{yi bd} |y_i| \leq \frac{T}{|x_1 x_2 x_3| \cdot |x_i|}\end{equation} 
for $i = 1,2,3$, and by Corollary 2 of \cite{Brow-HB} we then have the estimate
\[O \left(d(x_1 x_2 x_3) \left(\frac{T^{1/2}}{|x_1 x_2 x_3|} + O(1) \right) \right) \]
for the number of $\By \in \bZ_{\ne 0}^3$ satisfying (\ref{ht bd}) and (\ref{variety}) provided that the quadratic form $Q_\Bx$ has a rational zero. Otherwise it is clear that there will be no contribution. Thus we must estimate 
\[\sum_{\substack{1 \leq |x_1 x_2 x_3| \leq T^{1/2} \\ Q_\Bx \text{ has a rational zero}}} d(x_1 x_2 x_3).  \]
This is similar to the work of Guo in \cite{Guo}, except he counted with respect to the height $\lVert \Bx \rVert_\infty$. Nevertheless the techniques are similar, and again this may be of independent interest. \\

Next we must deal with the case when $|x_1 x_2 x_3| \geq T^{1/2}$. For this it suffices to observe from (\ref{yi bd}) that $|x_1 x_2 x_3| \geq T^{1/2}$ implies 
\[|y_1 y_2 y_3| \leq \frac{T^{3/2}}{(x_1 x_2 x_3)^2} \leq T^{1/2}.\]
We then treat (\ref{variety}) as a linear form $L_\By$ in $\Bx$. We use this to show that the contribution for each $\By$ is $O\left(T^{1/2} |y_1 y_2 y_3|^{-1} + 1 \right)$, which gives an acceptable contribution upon summing over $\By$. \\

For the lower bound, we first restrict $y_1, y_2, y_3 \in \bZ_{\ne 0}$ satisfying 
\[|y_1 y_2 y_3| \leq T^\delta \]
for some explicit $\delta > 0$ to be specified later. We note that to obtain the correct order of magnitude it is permissible to choose any $\delta > 0$. \\

Having fixed $\By = (y_1, y_2, y_3)$, we consider the simultaneous conditions (\ref{variety}) and (\ref{ht bd}). This gives rise to a binary form inequality of the shape 
\begin{equation} \label{bin form in} |x_1^2 x_2 (y_1^2 x_1 + y_2^2 x_2)| \leq T y_3^2 y_1^{-2}. 
\end{equation}
Because $|y_1 y_2 y_3|$ is small, we can count the number of solutions $\Bx$ to this inequality with reasonable precision. However, even with $|y_1 y_2 y_3|$ counting the number of solutions $\Bx$ with enough uniformity appears to still be a challenging task, because the binary form in (\ref{bin form in}) is singular. This difficulty is exacerbated by the fact that we will need to apply a square-free sieve eventually to produce triples $\Bx$ with each coordinate square-free. \\

To get around this issue, we simply count solutions to (\ref{bin form in}) with $x_1, x_2$ satisfying the inequalities 
\[|x_i y_i^2| \leq c_i T^{1/4} |y_1 y_2 y_3|^{1/2}, i = 1,2\]
for some positive numbers $c_1, c_2$. This has the effect that the long cusps inherent in (\ref{bin form in}) are removed, and reduces the problem to a more straightforward geometry of numbers question. 

\subsection{Upper bounds}
\label{crude}

To obtain upper bounds, it is crucial to view (\ref{variety}) as a plane in $x_1, x_2, x_3$  when $|y_1 y_2 y_3| \leq T^{1/2}$ and viewing (\ref{variety}) as a conic in $y_1, y_2, y_3)$ when $|x_1 x_2 x_3| \leq T^{1/2}$. We call the former the \emph{linear case} and the latter the \emph{quadratic case}. We proceed to deal with the linear case below. 

\subsubsection{The linear case} 

In this subsection we shall suppose that $|y_1 y_2 y_3| \leq T^{1/2}$ is fixed, and count the triples $(x_1, x_2, x_3)$ and $(y_1, y_2, y_3)$ for which (\ref{variety}) holds. \\

The key is the following lemma on counting points in sublattices of $\bZ^2$: 

\begin{lemma} \label{lat lem} Let $\Lambda \subset \bZ^2$ be a lattice. Then for all positive real numbers $R_1, R_2$ the number of primitive integral points $\Bx \in \Lambda$ satisfying $|x_i| \leq R_i, i = 1,2$ is at most $O \left(R_1 R_2/\det(\Lambda) + 1 \right)$. 
\end{lemma}

\begin{proof} If the rectangle $[-R_1, R_2] \times [-R_2, R_2]$ contains at least two primitive vectors in $\Lambda$, say $\Bx_1, \Bx_2$, then since this rectangle is convex it contains the parallelogram with end points $\pm \Bx_1, \pm \Bx_2$. The area of this parallelogram is at least as large as $\det \Lambda$, since the lattice spanned by $\Bx_1, \Bx_2$ is a sublattice of $\Lambda$. It thus follows that
\[R_1 R_2 \gg \det \Lambda.\]
Otherwise, the rectangle $[-R_1, R_1] \times [-R_2, R_2]$ contains at most one primitive vector in $\Lambda$. This completes the proof. 
\end{proof}
The strength of this lemma is that it gives a strong upper bound even in lopsided boxes. \\

Given (\ref{variety}), it follows that there is at least one $i \in \{2,3\}$ such that
\[|x_i y_i^2|/2  \leq x_1 y_1^2 \leq 2|x_i y_i^2|, \]
whence
\[\frac{x_1 y_1^2}{2 y_i^{-2}} \leq  |x_i| \leq \frac{2 x_1 y_1^2}{ y_i^{2}}.\]
Without loss of generality, we assume that this holds for $i = 2$. Suppose that $M_1 \leq x_1 < 2M_1$. By (\ref{ht bd}), we have
\[|x_3| \leq \frac{T}{|x_2 x_1^2 y_1^2|},\]
whence
\begin{align*} |x_3| & \leq T \cdot \frac{2 y_2^2}{(x_1 y_1^2)(x_1^2 y_1^2)}  \\
& \leq \frac{2y_2^2 T}{M_1^3 y_1^4}
\end{align*}
Applying Lemma \ref{lat lem} to the lattice defined by the congruence $y_1^2 x_1 - y_3^2 x_3 \equiv 0 \pmod{y_2^2}$ which has determinant equal to $y_2^2$, there are 
\[O \left(M_1 \cdot \frac{T y_2^2}{M_1^3 y_1^4} \cdot \frac{1}{y_2^2} + 1 \right) = O \left(\frac{T}{M_1^2 y_1^4} + 1 \right)\]
possibilities for $x_1, x_3$, which then determines $x_2 = (y_1^2 x_1 - y_3^2 x_3)/y_2^2$. Similarly, applying Lemma \ref{lat lem} to the lattice defined by $y_1^2 x_1 + y_2^2 x_2 \equiv 0 \pmod{y_3^2}$, with determinant equal to $y_3^2$, gives the estimate
\[O \left(M_1 \cdot \frac{y_1^2 M_1}{y_2^2} \frac{1}{y_3^2} + 1\right) = O \left(\frac{M_1^2 y_1^2}{y_2^2 y_3^2} + 1 \right) \]
for the number of $x_1, x_2$, which then also determine $x_3$. The two bounds coincide when 
\[M_1 = \frac{T^{1/4} |y_2 y_3|^{1/2}}{|y_1|^{3/2}}, \]
and we get the bound
\[O \left(\frac{T^{1/2} |y_2 y_3| y_1^2}{y_2^2 y_3^2 |y_1|^3} + 1 \right) = O \left(\frac{T^{1/2}}{|y_1 y_2 y_3|} + 1\right) \]
for the number of $x_1, x_2, x_3$ given $y_1, y_2, y_3$. Thus, we obtain an acceptable estimate whenever $|y_1 y_2 y_3| \ll T^{1/2}$, since
\begin{align*} \sum_{1 \leq |y_1 y_2 y_3| \leq T^{1/2}} \frac{T^{1/2}}{|y_1 y_2 y_3|} + 1 & \ll T^{1/2} \sum_{n \leq T^{1/2}} \frac{d_3(n)}{n} + \sum_{n \leq T^{1/2}} d_3(n)
\end{align*}
It is well-known that
\[\sum_{n \leq Z} d_3(n) = Z (\log Z)^2 + O(Z \log Z). \]
By partial summation, we have
\begin{align*} \sum_{n \leq Z} \frac{d_3(n)}{n} & = Z^{-1} \sum_{n \leq Z} d_3(n) + \int_1^Z \left(\sum_{n \leq t} d_3(n) \right) \frac{dt}{t^2} \\
& \ll (\log Z)^2 + \int_1^Z \frac{(\log t)^2 dt}{t} \\
& \ll (\log Z)^3
\end{align*}
It follows that 
\[T^{1/2} \sum_{n \leq T^{1/2}} \frac{d_3(n)}{n} + \sum_{n \leq T^{1/2}} d_3(n) \ll T^{1/2} (\log T)^3.\]

\subsubsection{The quadratic case} It remains to deal with the case when $|y_1 y_2 y_3| \gg T^{1/2}$, where we instead fibre over $\Bx$ and consider zeroes of the corresponding diagonal quadratic forms $Q_\Bx$. Since
\[|x_i y_i^2| \ll x_1 y_1^2 \]
for $i = 1,2$ by assumption, it follows that 
\[|x_1 x_2 x_3 y_1^2 y_2^2 y_3^2| \leq x_1^3 y_1^6,\]
hence
\[|y_1^2 y_2^2 y_3^2| \ll \frac{x_1^3 y_1^6}{x_1 |x_2 x_3|}.\]
If $|x_1 x_2 x_3| \gg T^{1/2}$, then 
\[x_1^3 y_1^6 \gg T^{3/2} \Leftrightarrow x_1 y_1^2 \gg T^{1/2}.\]
This implies that
\[|x_1 x_2 x_3| \cdot x_1 y_1^2 \gg T,\]
which violates (\ref{ht bd}) if the implied constants are sufficiently large. It thus follows that we must have $|x_1 x_2 x_3| \ll T^{1/2}$ in this case. \\

We now fix $x_1, x_2, x_3$ and consider (\ref{variety}) as a ternary quadratic form in $y_1, y_2, y_3$. We shall require the following version of Corollary 2 in \cite{Brow-HB}, which is an analogue of Lemma \ref{lat lem}: 

\begin{lemma} \label{quad lem} Let $x_1, x_2, x_3$ be pairwise co-prime square-free integers. Let $R_1, R_2, R_3$ be positive real numbers. Then the number of primitive solutions $y_1, y_2, y_3$ to the equation 
\[x_1 y_1^2 + x_2 y_2^2 = x_3 y_3^2\]
with $|y_i| \leq R_i$ is bounded by 
\[O \left( d(x_1 x_2 x_3) \left( \left(\frac{R_1 R_2 R_3}{|x_1 x_2 x_3|} \right)^{1/3} + 1 \right) \right).\]
\end{lemma}

Since $|x_i y_i^2| \ll x_1 y_1^2$ for $i = 1,2$, it follows that
\[|x_1 x_2 x_3 (x_i y_i^2)| \ll |x_1 x_2 x_3 (x_1 y_1^2)| \leq T\]
for $i = 1,2$, whence
\[|(x_1 y_1)^2 x_2 x_3|, |(x_2 y_2)^2 x_1 x_3|, |(x_3 y_3)^2 x_1 x_2| \ll T.\]
This implies that
\[(y_1 y_2 y_3)^2 (x_1 x_2 x_3)^4 \ll T^3,\]
hence
\[|y_1 y_2 y_3| \ll \frac{T^{3/2}}{(x_1 x_2 x_3)^2}.\]
Lemma \ref{quad lem} then implies that for fixed $x_1, x_2, x_3$ the number of primitive $\By = (y_1, y_2, y_3)$ satisfying (\ref{variety}) is 
\[O \left(d(x_1 x_2 x_3) \left(\frac{T^{1/2}}{|x_1 x_2 x_3|} + 1 \right) \right). \]
We now sum over primitive $\Bx \in \bZ^3$ satisfying $|x_1 x_2 x_3| \ll T^{1/2}$, with the property that the quadratic form $Q_\Bx$ given by (\ref{variety}) has a rational zero. By the Hasse-Minkowski theorem, this is tantamount to the form $Q_\Bx(\By) = x_1 y_1^2 + x_2 y_2^2 - x_3 y_3^2$ being everywhere locally soluble. The estimation of this is interesting on its own right and will be handled in a separate subsection.

\subsection{Counting soluble ternary quadratic forms} 

In this section, we consider the set
\[\S = \{(x_1, x_2, x_3) \in \bZ^3 : x_1, x_2, x_3 > 0, \gcd(x_1, x_2) = \gcd(x_1, x_3) = \gcd(x_2, x_3) = 1, \]
\[x_i \text{ square-free for } i = 1,2,3, x_1 y_1^2 + x_2 y_2^2 - x_3 y_3^2 \text{ is everywhere locally soluble}\}.\]
By a well-known theorem of Legendre (see \cite{Guo}) the indicator function for $\S$ is given by
\begin{equation} f_\S(x_1, x_2, x_3) = \left(2^{-\omega(x_1)} \sum_{a_1 | x_1} \left(\frac{x_2 x_3}{a_1} \right) \right) \left(2^{-\omega(x_2)} \sum_{a_2 | x_2} \left(\frac{x_1 x_3}{a_2} \right) \right) \left(2^{-\omega(x_3)} \sum_{a_3 | x_3} \left(\frac{-x_1 x_2}{a_3} \right) \right).
\end{equation}
We will now combine the ideas given in \cite{Guo} and those in \cite{FK}. \\

Put 
\begin{align*} \S(X) & = \sum_{1 \leq x_1 x_2 x_3 \leq X}  \sum_{(x_1, x_2, x_3) \in \S} \frac{d(x_1 x_2 x_3)}{x_1 x_2 x_3} \\ & = \sum_{1 \leq |x_1 x_2 x_3| \leq X} \frac{d(x_1 x_2 x_3)}{x_1 x_2 x_3}  f_\S(x_1, x_2, x_3). \end{align*}
Since $x_1, x_2, x_3$ are pairwise coprime and square-free, it follows that
\[d(x_1 x_2 x_3) = 2^{\omega(x_1 x_2 x_3)} = 2^{\omega(x_1)} \cdot 2^{\omega(x_2)} \cdot 2^{\omega(x_3)}, \]
where $\omega(n)$ is the number of distinct prime factors of $n$. It follows that
\begin{equation} \S(X) = \sum_{1 \leq x_1 x_2 x_3 \leq X} \frac{2^{\omega(x_1 x_2 x_3)}}{x_1 x_2 x_3} f_\S(x_1, x_2, x_3) 
\end{equation}
\[ = \sum_{1 \leq x_1 x_2 x_3 \leq X} \frac{1}{x_1 x_2 x_3} \left(1 + \left(\frac{x_2 x_3}{x_1} \right) \left(\frac{x_1 x_3}{x_2} \right) \left(\frac{-x_1 x_2}{x_3} \right)  + \sum_g g(x_1, x_2, x_3)\right),\]
where $g$ expresses a product of Jacobi symbols. The sum 
\begin{equation} \label{S1X} \S_1(X) = \sum_{1 \leq |x_1 x_2 x_3| \leq X} \frac{1}{x_1 x_2 x_3} \left(1 + \left(\frac{x_2 x_3}{x_1} \right) \left(\frac{x_1 x_3}{x_2} \right) \left(\frac{-x_1 x_2}{x_3} \right)  \right) \end{equation}
is expected to contribute the main term while the sum 
\begin{equation} \label{char sum} \S_2(x) = \sum_{1 \leq x_1 x_2 x_3 \leq X} \frac{1}{x_1 x_2 x_3} \sum_g g(x_1, x_2, x_3) \end{equation}
is expected to be negligible, due to the cancellation of characters. \\

By partial summation, we obtain:
\begin{equation} \label{par sum} \S_i(X) = \frac{1}{X} \Sigma_i(X) + \int_1^X \Sigma_i(t) \frac{t}{t^2},
\end{equation}
where
\[\Sigma_1(X) = \sum_{\substack{1 \leq |x_1 x_2 x_3| \leq X \\ x_1 x_2 x_3 \text{ square-free}  \\ Q_{(x_1, x_2, x_3)} \text{ is soluble} }} \left(1 + \left(\frac{x_2 x_3}{x_1} \right) \left(\frac{x_1 x_3}{x_2} \right) \left(\frac{x_1 x_2}{x_3} \right) \right) \]
and
\[\Sigma_2(X) = \sum_{1 \leq x_1 x_2 x_3 \leq X} \sum_g g(x_1, x_2, x_3). \]

Our situation differs from that of Guo in \cite{Guo} since we are counting over triples with $|x_1 x_2 x_3| \leq X$ rather than $\max\{|x_1|, |x_2|, |x_3|\} \leq X$, which introduces some difficulties. However, this is exactly analogous to the situation encountered by Fouvry and Kluners in \cite{FK}. \\

Our key proposition will be:  

\begin{proposition} \label{quad form prop} We have the asymptotic upper bound
\[\S(X) = O\left((\log X)^3 \right).\]
\end{proposition} 
In fact we can refine Proposition \ref{quad form prop} to give an asymptotic formula, but this is unnecessary for our purposes. \\

We 

We proceed to prove Proposition \ref{quad form prop} in the remainder of the section. We begin by showing that triples $(x_1, x_2, x_3)$ with $\mu^2(x_1 x_2 x_3) = 1$ and $\omega(x_1 x_2 x_3)$ large contribute negligibly. To wit, put
\[\S_2^{(r)}(X) = \sum_{\substack{1 \leq x_1 x_2 x_3 \leq X \\ \omega(x_1 x_2 x_3) = r}} \frac{1}{x_1 x_2 x_3} \sum_g g(x_1, x_2, x_3).\]
By the triangle inequality, it is clear that 
\[\left \lvert \S_2^{(r)}(X) \right \rvert \ll \sum_{\substack{n \leq X \\ \mu^2(n) = 1, \omega(n) = r}} \frac{d_3(n)}{n}. \]
By partial summation, we have
\[\sum_{\substack{n \leq X \\ \mu^2(n) = 1, \omega(n) = r}} \frac{d_3(n)}{n} = X^{-1} \sum_{\substack{n \leq X \\ \mu^2(n) = 1, \omega(n) = r}} d_3(n) + \int_1^X \left(\sum_{\substack{n \leq t \\ \mu^2(n) = 1, \omega(n) = r}} d_3(n) \right)\frac{dt}{t^2}. \]
To estimate the latter sum, we will need the following result, which is Lemma 11 in \cite{FK}:

\begin{lemma} There exists an absolute constant $B_0 \geq 1$ such that for every $r \geq 0$, we have
\[|\{n \leq X : \omega(n) = r, \mu^2(n) = 1\}| \leq B_0 \cdot \frac{X}{\log X} \cdot \frac{(\log \log X + B_0)^r}{r!} \]
\end{lemma}

Applying the lemma, we have for $\Omega = 30 (\log \log X + B_0)$
\begin{align*} \sum_{\substack{n \leq X \\ \mu^2(n) = 1, \omega(n) \geq \Omega}} d_3(n) & \ll \frac{X}{\log X} \sum_{r \geq \Omega} 3^r \cdot \frac{(\log \log X + B_0)^r}{r!} \\
& \ll \frac{X}{\log X} \sum_{r \geq \Omega} \left(\frac{3e (\log \log X + B_0)}{r} \right)^r \\
& \ll \frac{X}{\log X} \sum_{r \geq \Omega} \left(\frac{3e}{10} \right)^r,
\end{align*}
the final sum a convergent geometric series. Hence
\[\sum_{r \geq \Omega} \left(\frac{3e}{10} \right)^r \ll \left(\frac{3e}{10} \right)^{\Omega} \ll \frac{1}{\log X}. \]
We thus conclude that 
\begin{align} \sum_{r \geq \Omega} \left \lvert S_2^{(r)}(X) \right \rvert & \ll 1 + (\log X)^{-2} + \int_1^X \frac{dt}{t (\log t)^2}  \\
& = O(1) \notag 
\end{align}
and is thus negligible. \\

Note that $x_1, x_2, -x_3$ cannot all be the same sign, otherwise (\ref{variety}) will only have a trivial real solution. Hence the signs of $(x_1, x_2, x_3)$ must be $(+, +, +)$, or $(+, -, +)$, since we assumed $x_1 > 0$ and $x_1 y_1^2 \geq |x_2 y_2^2|$. By rearranging, we must thus assume $x_1, x_2, x_3 > 0$. \\

We then expand (\ref{char sum}) by writing $x_i = x_{i1}x_{i2}$ for $i = 1,2,3$, and
\[\sum_{\substack{1 \leq x_1 x_2 x_3 \leq X \\ \mu^2(x_1 x_2 x_3) = 1}} \sum_g g(x_1, x_2, x_3) = \sum_{\substack{(x_{11}x_{12})(x_{21}x_{22})(x_{31}x_{32})\leq X \\ 1 < x_{i1} < x_i \text{ for } 1 \leq i \leq 3}} \left(\frac{x_{21}x_{22} x_{31}x_{32} }{x_{11}} \right) \left( \frac{x_{11} x_{12} x_{31} x_{32}}{x_{21}} \right) \left(\frac{x_{11} x_{12} x_{21} x_{22}}{x_{31}} \right). \]
We now follow the strategy outlined in \cite{FK} and break up the set 
\[\{(x_{11}, x_{12}, x_{21}, x_{22}, x_{31}, x_{32}) \in \bN^6 : x_{11} x_{12} x_{21} x_{22} x_{31} x_{32} \leq X \} \]
by restricting the $x_{ij}$'s to intervals of the form 
\[[A_{ij}, \Delta A_{ij}),\]
where 
\[\Delta = 1 + (\log X)^{-3}.\] 
For a given $\BA = (A_{11}, A_{12}, A_{21}, A_{22}, A_{31}, A_{32})$, put 
\[\S_2(X; \BA) = \sum_{\substack{x_{ij} \in [A_{ij}, \Delta A_{ij}) \\ \mu^2(x_{11} x_{12} x_{21} x_{22} x_{31} x_{32}) = 1 \\ \prod_{i,j} x_{ij} \leq X }} \left(\frac{x_{21}x_{22} x_{31}x_{32} }{x_{11}} \right) \left( \frac{x_{11} x_{12} x_{31} x_{32}}{x_{21}} \right) \left(\frac{x_{11} x_{12} x_{21} x_{22}}{x_{31}} \right). \]

We then have the following lemma:

\begin{lemma} We have the bound
\[\sum_{\prod A_{ij} \geq \Delta^{-6} X} \left \lvert \S_2(X; \BA) \right \rvert = O \left(X (\log X)^{-1} \right). \]
\end{lemma}

\begin{proof} We have
\begin{align*} \sum_{\prod A_{ij} \geq \Delta^{-6} X} \left \lvert \S_2(X; \BA) \right \rvert & \leq \sum_{\substack{\Delta^{-6} X \leq n \leq X \\ \mu^2(n) = 1}} d_3(n) \\
& \ll \sum_{\Delta^{-6} X \leq n \leq X} 3^{\omega(n)} \\
& \ll (1 - \Delta^{-6} ) X (\log X)^2. 
\end{align*}
By Taylor's theorem, we have
\[\Delta^{-6} = (1 + (\log X)^{-3})^{-6} = 1 - 6 (\log X)^{-3} + O \left((\log X)^{-6} \right).\]
The proof then follows.
\end{proof}

To proceed, we shall require the following well-known lemma regarding character sums: 

\begin{lemma}[Double Oscillation Lemma] \label{double} Let $\{\alpha_n\}, \{\beta_m\}$ be two sequences of complex numbers with each term having absolute value bounded by $1$. Let $M,N$ be positive real numbers. Then  we have
\[\sum_{m \leq M} \sum_{n \leq N} \alpha_m \beta_n \mu^2(2m) \mu^2(2n) \left(\frac{m}{n} \right) \] 
\[ \ll \min \left\{ \left(M^{-1/2} + (N/M)^{-1/2}  \right), \left(N^{-1/2} + (M/N)^{-1/2} \right) \right\}  \]
and for every $\ep > 0$, 
\[\sum_{m \leq M} \sum_{n \leq N} \alpha_m \beta_n \mu^2(2m) \mu^2(2n) \ll_\ep MN \left(M^{-1/2} + N^{-1/2} \right) (MN)^\ep \] 
\end{lemma}

We will also need the following variant of the Siegel-Walfisz theorem: 

\begin{lemma} \label{S-W} Let $\chi_q$ be a primitive character modulo $q \geq 2$. Then for every $A > 1$ we have
\[\sum_{Y \leq p \leq X} \chi_q(p) = O_A \left(\sqrt{q} \cdot X (\log X)^{-A} \right) \]
uniformly for $X \geq Y \geq 2$. 
\end{lemma}

We now consider, as in \cite{FK}, the quantities
\begin{equation} X^\dagger = (\log X)^9, X^\ddagger = \exp\left( (\log X)^{1/8} \right).
\end{equation}
We now consider those $\BA$ with the property that at most $2$ entries larger than $X^\ddagger$. We dissect the sum according to the number $r \leq 2$ of terms $A_{ij}$ greater than $X^\ddagger$. Let $n$ be the product of those $x_{ij}$ which are larger than $X^\ddagger$, and $m$ the product of the remaining ones. We sum over $\BA$ with such properties to obtain
\begin{align*} \sideset{}{^{(2)}} \sum_{\BA}  |\S_2(X; \BA)| & \leq \sum_{r \leq 2} \sum_{m \leq (X^\ddagger)^{6-r}} \mu^2(m) d_{6-r}(m) \sum_{n \leq X/m} \mu^2(n) d_r(n) \\
& \ll \sum_{r \leq 2} \sum_{m \leq (X^\ddagger)^{6-r}} \mu^2(m) d_{6-r}(m) \left(\frac{X}{m} \right) (\log X)^{r-1}  \\ 
& \ll X \left(\sum_{r \leq 2} (\log X)^{r-1} \right) \left(\sum_{m \leq (X^\ddagger)^{6}} \frac{d_{6}(m)}{m} \right) \\
& \ll X (\log X)\left(\log \exp \left((\log X)^{1/8} \right) \right)^7 \\
& \ll X (\log X)^{15/8}. 
\end{align*}
This is sufficiently small for our purposes. \\ 

We may now assume that $A_{ij} \geq X^\ddagger$ for at least three pairs $i,j$ with $1 \leq i \leq 3, 1 \leq j \leq 2$. We now suppose that there exist $a \ne b$ such that 
\[A_{a,2}, A_{b,1} \geq X^\dagger. \]
The sum over $\BA$ satisfying these properties can be bounded by
\begin{align*} \sum_{\BA} \left \lvert \S_2(X; \BA) \right \rvert & \leq \sum_{x_{ij}, (i,j) \ne (a,2), (b,1)} \prod_{(i,j) \ne (a,2), (b,1)} \left \lvert \sum_{x_{a,2}} \sum_{x_{b,1}} \alpha_{(a,2)} \beta_{(b,1)} \left(\frac{x_{a,2}}{x_{b,1}} \right) \right \rvert,
\end{align*}
where $\alpha, \beta$ have modulus at most one. Lemma \ref{double} then applies, and since our variables $x_{a,2}, x_{b,1}$ range over intervals exceeding $X^\dagger$ in length, it follows that 
\[|\S_2(X;\BA)| \ll \left(\prod_{(i,j) \ne (a,2), (b,1)} A_{ij} \left(A_{a,2} A_{b,1} \left(A_{a,2}^{-1/3} + A_{b,1}^{-1/3} \right) \right) \right) \ll X (X^\dagger)^{-1/3} = O \left(X (\log X)^{-3} \right), \]
which is again enough. \\

Next consider the family where the two previous conditions do not hold, and in addition there exist $a \ne b$ such that $2 \leq A_{b,1} \leq X^\dagger$ and $A_{a,2} > X^\ddagger$. Under these conditions, we see that
\[|\S_2(X; \BA)| \ll \sum_{x_{ij}, (i,j) \ne (a,2), (b,1)} \sum_{x_{a,2}} \left \lvert \sum_{x_{b,1}} \mu^2 \left(\prod_{(i,j) \ne (a,2), (b,1)} x_{ij} \right) \left(\frac{x_{a,2}}{x_{b,1}} \right) \right \rvert,   \]
where $A_{ij} \leq x_{ij} \leq \Delta A_{ij}$ and $\omega(x_{ij}) \leq \Omega$ for $1 \leq i \leq 3, 1 \leq j \leq 2$. Now put $\ell = \omega(x_{a,2})$, writing 
\[x_{a,2} = p_1 \cdots p_\ell \]
with $p_1 < p_2 < \cdots < p_\ell$ we obtain
\[|\S_2(X; \BA)| \ll \sum_{\substack{x_{ij} \\ (i,j) \ne (a,2), (b,1)}} \sum_{x_{b,1}} \sum_{0 \leq \ell \leq \Omega} \left \lvert \sum_{\omega(x_{a,2}) = \ell} \mu^2\left(\prod_{i,j} x_{ij} \right) \left(\frac{x_{a,2}}{x_{b,1}} \right) \right \rvert,\]
the inner sum being bounded by 
\[\sum_{p_1 \cdots p_{\ell-1}} \left \lvert \sum_{p_\ell} \left(\frac{p_\ell}{x_{b,1}} \right) \right \rvert \] 
and $p_1, \cdots, p_\ell$ satisfy $A_{a,2} \leq p_1 \cdots p_{\ell} \leq \Delta A_{a,2}$. Note that
\[p_\ell \geq A_{a,2}^{1/\ell} \geq \exp \left((\log X)^{1/9} \right).  \]
We may now apply Lemma \ref{S-W} to obtain the bound 
\[\left \lvert \sum_{p_\ell} \left(\frac{p_\ell}{x_{b,1}} \right) \right \rvert \ll_A A_{b,1}^{1/2} \frac{A_{a,2}}{p_1 \cdots p_{\ell-1}} (\log X)^{-A/9} + \Omega, \]
with $A$ arbitrarily large. Note that $p_1 \cdots p_{\ell-1} \leq X$, hence
\[\sum_{p_1 \cdots p_{\ell-1} \leq X} (p_1 \cdots p_{\ell-1})^{-1} \ll \sum_{n \leq X} \frac{1}{n} \ll \log X. \]
Hence 
\[\sideset{}{^{(3)}} \sum_\BA |\S_2(X; \BA)| \ll A_{b,1}^{1/2}\prod_{i,j} A_{ij} (\log X)^{-A/9 + 1} \ll X (\log X)^{-A/9 + 11/2}. \]
Choosing $A$ large shows that this contribution is negligible. \\

The remaining case can be summarized by the following properties:
\begin{enumerate}
    \item $\prod_{i,j} A_{ij} \leq \Delta^{-6} X$; 
    \item $A_{ij} \geq X^\ddagger$ for at least three pairs of indices $(i,j)$; 
    \item If $A_{ij}, A_{k\ell} \geq X^\dagger$ then $j = \ell$; 
    \item If $A_{ij} \leq A_{k \ell}$ with $j \ne \ell$, then either $A_{ij} = 1$ or $2 \leq A_{ij} \leq X^\dagger$ and $A_{k \ell} < X^\ddagger$. 
\end{enumerate}

We now show that the second option in (4) cannot happen. This will imply that we have accounted for all possibilities for (\ref{char sum}), and hence reduced our problem to estimating $\S_1(X)$. \\

Suppose, without loss of generality, that $2 \leq A_{11} \leq X^\dagger$ and $A_{22} < X^\ddagger$. Since $A_{ij} \geq X^\ddagger$ for at least three pairs of indices $(i,j)$, one of $A_{12}$ or $A_{32}$ must exceed $X^\ddagger$. We then have $A_{11} \leq X^\dagger$ and $A_{32}$, say, exceeds $X^\ddagger$, which means that our earlier estimation covers this case. \\

The upshot now is that
\begin{equation} \Sigma_2(X) \ll_A X (\log X)^{15/8}
\end{equation}
for some $\kappa(A) > 0$. It follows from (\ref{par sum}) that
\begin{align*} \S_2(X) & = X^{-1} \Sigma_2(X) + \int_1^X \Sigma_2(t) \frac{dt}{t^2} \\
& \ll (\log X)^{15/8} + \int_1^X \frac{(\log t)^{15/8} dt}{t} \\
& = (\log X)^{23/8},
\end{align*}
which is sufficiently small for our purposes. \\

Finally, we may evaluate the main term, which is given by (\ref{S1X}). By the triangle inequality, we have
\[\S_1(X) \ll \sum_{x_1 x_2 x_3 \leq X} \frac{1}{x_1 x_2 x_3} = \sum_{n \leq X} \frac{d_3(n)}{n}\]
which is $O((\log X)^3)$. This completes the proof of the Proposition. 

\subsection{Lower bounds}

For the lower bound, we shall assume 
\[1\leq |y_1 y_2 y_3| \leq T^\delta\]
where $\delta$ is some explicit positive number which we shall specify later. We then consider $x_1, x_2$ satisfying 

\begin{equation} \label{box bd} |x_i y_i^2| \leq c_i T^{1/4} |y_1 y_2 y_3|^{1/2}, i = 1,2\end{equation} 
where $c_1, c_2$ are two small positive numbers. Note that 
\[|x_3 y_3^2| = |x_1 y_1^2 + x_2 y_2^2| \leq |x_1 y_1^2| + |x_2 y_2^2| \leq (c_1 + c_2) T^{1/4} |y_1 y_2 y_3|^{1/2}, \]
whence
\[|x_1 x_2 x_3| (y_1 y_2 y_3)^2 \leq c_3 T^{3/4} |y_1 y_2 y_3|^{3/2} \]
where $c_3 = c_1 c_2 (c_1 + c_2)$. Thus 
\[|(x_1 y_1^2) x_1 x_2 x_3| \leq (c_1 T^{1/4} |y_1 y_2 y_3|^{1/2}) (c_3 T^{3/4} |y_1 y_2 y_3|^{-1/2}) \]
which is less than $T$ provided that $c_1 c_3 \leq 1$. Therefore every pair $(x_1, x_2)$ satisfying (\ref{box bd}) with $x_1, x_2$ both square-free and $x_3 = (y_1^2 x_1 + y_2^2 x_2) y_3^{-2} \in \bZ$ square-free will contribute to $N(T)$. \\

We now count pairs $(x_1, x_2)$ such that
\begin{enumerate}
    \item $(x_1, x_2)$ satisfies (\ref{box bd}); 
    \item $\gcd(x_1, x_2) = 1$; 
    \item $x_1, x_2$ are square-free; and
    \item $y_1^2 x_1 + y_2^2 x_2 \equiv 0 \pmod{y_3^2}$, $(y_1^2 x_1 + y_2^2 x_2)y_3^{-2}$ is square-free. 
\end{enumerate}
For each prime $p$, we interpret conditions (2) to (4) modulo $p^2$. Condition (2) is the assertion that $p | x_1 \Rightarrow p \nmid x_2$, Condition (3) is the assertion that for all primes $p$ we have $p^2 \nmid x_1, x_2$, and Condition (4) is stating $y_3^2 | y_1^2 x_1 + y_2^2 x_2$, and if $p^{s} || y_3$, then $p^{2s + 2} \nmid y_1^2 x_1 + y_2^2 x_2$. Let 
\[\rho_\By(m) = \# \{(x_1, x_2) \pmod{m} : (2) \text{ to } (4) \text{ holds for all } p | m\}. \]
It is apparent that $\rho_\By(\cdot)$ is multiplicative. Put 
\[N^\ast(\By; T) = \# \{(x_1, x_2) \in \bZ^2 : (1) \text{ to } (4) \text{ hold} \} \]
and 
\[N_b^\ast(\By; T) = \#\{(x_1, x_2) \in \bZ^2 : (\ref{box bd}) \text{ holds, }(2) \text{ to } (4) \text{ holds mod } b \}\]
By standard arguments, we have 
\[N^\ast(\By; T) = \prod_{p \leq Y} \left(1 - \frac{\rho_\By(p^{2k})}{p^{2k}} \right) \frac{T^{1/2} }{|y_1 y_2 y_3|} + O \left(\sum_{Y < p < T^{1/8} |y_1 y_2 y_3|^{1/4} \max\{|y_1|^{-1},|y_2|^{-1}\}} \left(\frac{T^{1/2}}{p^2 |y_1 y_2 y_3|} + 1 \right) \right), \]
the error term being bounded by 
\[O \left(\frac{T^{1/2}}{Y |y_1 y_2 y_3|} + \frac{T^{1/8} |y_1 y_2 y_3|^{1/2}}{\min\{|y_1|, |y_2|\}} \right). \]
Since $|y_1 y_2 y_3| \leq T^\delta$, we obtain an acceptable error term provided that $\delta < 1/4$. This shows that 
\[N(T) \gg \sum_{1 \leq |y_1 y_2 y_3| \leq T^\delta} N^\ast(\By; T) \gg \sum_{1 \leq |y_1 y_2 y_3| \leq T^\delta} \frac{T^{1/2}}{|y_1 y_2 y_3|}.\]
Since
\[\sum_{1 \leq |y_1 y_2 y_3| \leq Z} |y_1 y_2 y_3|^{-1} \gg \sum_{n \leq Z} d_3(n) n^{-1} \gg  (\log Z)^3, \]
this confirms the lower bound. 

\section*{Appendix: Prolegomena to a theory of heights on algebraic stacks}

In this section we give a hint of the flavor of heights on algebraic stacks, leaving the details to the forthcoming (\cite{ESD-B}). Our purpose here is to highlight the difficulties faced when attempting to define heights on algebraic stacks, while instilling a sense of excitement about the future possibilities of this theory.\\  

In order to recognize recognize the Manin and Malle conjectures in a unified theoretical framework, (\cite{ESD-B}) developed a theory of heights on algebraic stacks. We now explain why such a theory naturally arises when trying to unify Manin and Malle. In Manin's conjecture one counts points using the anti-canonical height, therefore Manin's conjecture is intimately related to the theory of heights on projective varieties. We take the theory of heights as our starting point and naively attempt to use this approach for the Malle conjecture. To approach Malle's conjecture in a similar manner one might attempt to endow the collection of $G$-extensions of a number field $K$ with the structure of an algebraic variety, and then count points on this variety using the height machine. Concretely one way to realize this approach would be as follows.

\begin{enumerate}
    \item Construct a projective variety $\mathcal{X}_G$ such that the $K$-rational points $\mathcal{X}_G(K)$ correspond bijectively to extensions of $K$ with Galois group $G$. 
    \item Find a good height function $h_{\mathcal{X}_G}$ on $\mathcal{X}_G$ that relates the height of a point $P\in \mathcal{X}_G(K)$ to the discriminant of the Galois extension associated to $P$. 
    \item  Use the theory heights and the geometry of $\mathcal{X}_G$ to count points on $\mathcal{X}_G(K)$ and thus count $G$-extensions of $K$.
\end{enumerate}

One runs into problems immediately because the set of $G$ extensions of a number field $K$ cannot naturally be realized as the set of points of a scheme. Indeed if a collection of objects can be \emph{naturally} realized as the set of points of a scheme, then the objects in question must have no non-trivial automorphisms. On the other hand, $G$-extensions of $K$ \emph{always} have non-trivial automorphisms given by the Galois group $G$. As is well known at this point, an appropriate setting  for moduli problems with automorphisms is given by the theory of algebraic stacks introduced by Deligne and Mumford in their foundational study of the moduli space of curves. An important basic example of an algebraic stack is the $\emph{classifying stack}$  $BG$ of a finite group $G$. As the name suggests the classifying stack is a moduli space, whose points correspond to $G$-torsors. When $K$ is a number field the $K$-points of $BG$ correspond to $G$-torsors over $K$ which are precisely the Galois extensions of $K$ with Galois group $G$. Thus the theory of algebraic stacks provides a moduli space $BG$ whose points correspond to the arithmetic objects of interest, namely $G$-extensions of $K$, achieving the first task in the outline above.\\

A more serious problem arises in the second point. Let $G=\mathbb{Z}/2\mathbb{Z}$ and consider the classifying stack $B(\bZ/2\bZ)$. A theory of heights parallel to that of schemes would suggest that a height function on $B(\bZ/2\bZ)$ should correspond to a line bundle on $\mathcal{L}$ on $B(\bZ/2\bZ)$. In general a vector bundle of rank $r$ on $BG$ corresponds to an $r$-dimensional representation of $G$. Taking $r=1$ we see that $B(\bZ/2\bZ)$ has line bundles corresponding to the characters of the group $\bZ/2\bZ$. There are precisely two characters of $\bZ/2\bZ$, the trivial character $\textnormal{triv}(\epsilon)=1$ for $\epsilon\in \bZ/2\bZ$ corresponding to the trivial line bundle and the sign representation $\textnormal{sgn}(\epsilon)=\textnormal{sgn}(\epsilon)$ where $\epsilon\in \bZ/2\bZ$ is considered as an element of the permutation group $S_2$. Thus we expect that up to bounded functions that there are two height functions on $B(\bZ/2\bZ)$. The height $h_{\textnormal{sgn}}$ associated the sign representation and the height function $h_{\textnormal{triv}}$ associated to the trivial line bundle. The functoriality property of the height machine tells us that for any line bundle $\L$ we should expect that $h_{\L^{\otimes n}}=nh_{\L}+O(1)$. Taking $\L=\textnormal{triv}$ we see that $\textnormal{triv}^{\otimes n}=\textnormal{triv}$ and so
\[h_{\textnormal{triv}}=h_{\textnormal{triv}^{\otimes n}}=nh_{\textnormal{triv}}+O(1)\] for all $n$.  Thus $h_{\textnormal{triv}}$ must be a bounded function. On the other hand as the tensor product of characters is the function given by taking the product of the characters we have that $\textnormal{sgn}\otimes \textnormal{sgn}=\textnormal{sgn}^2=\textnormal{triv}$
the trivial representation. Thus functoriality of heights suggests that we should have an equality
\[h_{\textnormal{triv}}=h_{\textnormal{sgn}\otimes\textnormal{sgn}}=2h_{\textnormal{sgn}}+O(1).\]
Consequently $2h_{\textnormal{sgn}}$ would be the trivial height meaning that $h_{\textnormal{triv}}$ would be some bounded function. Clearly this is not satisfactory. For example, no Northcott property can be expected for any height function as $B(\bZ/2\bZ)(\bQ)$ corresponds to  quadratic extensions of $\bQ$ of which there are infinitely many, but the argument above suggests that the points of $B(\bZ/2\bZ)$ is a set of bounded height. \\

Despite these obstacles, J.~Ellenberg, M.~Satriano, and D.~Zuerick-Brown realized that one may develop a theory of heights on algebraic stacks at the cost of losing the functoriality properties of the height machine. Their theory associates a height function $h_\E$ to each vector bundle $\E$ on an algebraic stack $\X$. If $\X$ is taken to be a scheme then the associated height function is the classical height $h_{\det\E}$ associated to the determinant line bundle $\det\E$. Thus this new theory of heights recovers the classical theory as a special case. Furthermore, given a suitable algebraic stack $\X$ and a vector bundle $\E$ on $\X$ such that the associated height $h_\E$ satisfies a Northcott property one has a conjecture (\cite[Main Conjecture]{ESD-B}) that predicts the asymptotic behaviour of points of bounded height on $\X$ with respect to $h_\E$. Given a Fano variety with an ample line bundle $\L$ (\cite[Main Conjecture]{ESD-B}) applied to $(X,h_{\L})$ is the Manin conjecture, while (\cite[Main Conjecture]{ESD-B}) applied to $(BG,h_{\textnormal{regular}})$ recovers the Malle conjecture. Here $h_\textnormal{regular}$ is the height obtained from the vector bundle associated to the regular representation of the finite group $G$.\\ 

To sum up, there is a theory of heights on algebraic stacks that is mostly unexplored. Given an algebraic stack $\mathcal{X}$ and nice enough vector bundle $\mathcal{E}$ on $\mathcal{X}$ there is a conjectur (\cite[Conjecture]{ESD-B}), that describes the behavior of the pair $(\mathcal{X},\mathcal{E})$. Furthermore, (\cite[Main Conjecture]{ESD-B}) specializes to the Manin and Malle conjecture by choosing $\mathcal{X}$ and $\mathcal{E}$ appropriately. The full conjecture is almost completely open, though in (\cite{ESD-B}) some additional cases will be considered.

\end{document}